\documentclass[11pt, english, a4paper]{article}
\usepackage[latin1]{inputenc}
\usepackage[T1]{fontenc}
\usepackage{babel, graphicx, textcomp, varioref, amsfonts}
\usepackage{amsthm}
\usepackage{amsmath}

\newtheorem{thm}{Theorem}[section]

\newtheorem{lemma}[thm]{Lemma}
\newtheorem{prop}[thm]{Proposition}

\newtheorem{remark}[thm]{Remark}

\date{}

\title{A Note on the Malliavin Differentiability of
One-Dimensional Reflected Stochastic Differential Equations with
Discontinuous Drift}

\author{ Torstein Nilssen \thanks{Department of Mathematics, University of Oslo, Moltke Moes vei 35, P.O. Box 1053 Blindern, 0316 Oslo, Norway.
E-mail: torsteka@math.uio.no. 
Funded by  Norwegian Research Council (Project 230448/F20).} 
 , 
Tusheng Zhang \thanks{School of Mathematics, University of Manchester, Oxford Road, Manchester M13 9PL, England. 
E-mail: tzhang@maths.man.ac.uk . }}

\begin{document}

\maketitle

\begin{abstract}
We consider a one-dimensional Stochastic Differential Equation with reflection where we allow the drift to be merely bounded and measurable. It is already known that such equations have a unique strong solution, see \cite{Tusheng94}. In \cite{PMNPZ} and \cite{MNP} it is shown that non-reflected SDE's with discontinuous drift possess more regularity than one could expect, namely they are Malliavin differentiable and weakly differentiable w.r.t. the initial value. See also \cite{FedrizziFlandoli} for a different technique. 

In this paper we use the approach of \cite{PMNPZ} and \cite{MNP} to show that similar results hold for one-dimensional SDE's with reflection. We then apply the results to get a Bismut-Elworthy-Li formula for the corresponding Kolmogorov equation.

\end{abstract}

\section{Introduction and Main Result}
Let $(\Omega, \mathcal{F}, P)$ be a complete probability space
with a filtration $\{\mathcal{F}_t\}$ satisfying the usual conditions. Let
$B_t$ be a standard $\mathcal{F}_t$-Brownian motion on the
space.

We consider a stochastic differential equation with
reflecting boundary:
\begin{equation} \label{mainSDE}
\left\{ \begin{array}{ll}
dX_t & = b(X_t)dt + \sigma(X_t)dB_t + dL_t \\
X_t & \geq 0 \\
L_t & = 1_{\{0\}}(X_t) dL_t ,\\
\end{array}
\right.
\end{equation}

where $b:\mathbb{R} \rightarrow \mathbb{R}$ is a bounded and
measurable function and $\sigma: \mathbb{R} \rightarrow \mathbb{R}$ is a continuously differential bounded function, bounded away from zero. This equation has a unique strong solution
as proved in \cite{Tusheng94}, namely there exists a pair
$(X,L)$ of processes such that
\begin{itemize}
\item
$X_t$ is $\mathcal{F}_t$-adapted, $X_t \geq 0$ for all $t$.

\item
$L_t$ is $\mathcal{F}_t$-adapted, continuous, non-decreasing and
such that
$$
L_0 = 0 , \, \, L_t = \int_0^t 1_{\{0\}}(X_s) dL_s,
$$

\item
$X_t = x + \int_0^t b(X_s)ds+ \int_0^t \sigma(X_s) B_s + L_t , \, \, P-a.s.$

\end{itemize}

We call $x \geq 0$ the initial value of the equation.

For simplicity, we shall consider the equation defined on $t \in [0,1]$.

The aim of this paper is to show the following

\begin{thm} \label{MalliavinThm}
Assume $b$ is bounded and measurable, $\sigma \in C_b^1(\mathbb{R})$ and there exists $\delta >0$ such that $|\sigma(x)| \geq \delta$ for all $x$.

Then the strong solution to (\ref{mainSDE}) is Malliavin-differentiable, i.e. for a fixed $t \in [0,1]$, we have $X_t \in \mathbb{D}^{1,2}$.
\end{thm}

The outline of this paper is as follows:
The rest of this section is devoted to the proof of Theorem \ref{MalliavinThm}.
In Section \ref{spatialRegularity} we study how the solution depends on the initial value $x \geq 0$. We then use this to study the corresponding Kolmogorov equation and obtain a Bismut-Elworthy-Li formula in Section \ref{bismutFormula}.
Section \ref{skorohodEquation} is the Appendix where we include an approximation of the Skorohod equation.

\bigskip

We return to the proof of Theorem \ref{MalliavinThm} which is divided into three steps. In the two first steps we
consider (\ref{mainSDE}) with a drift $b \in C^1_b(\mathbb{R})$ such that $b(0) = 0$.
In \emph{step 1} we introduce an approximation of the solution in terms of an ordinary SDE, i.e. not reflected. 

In \emph{step 2} we use the approximation from step 1 to find bounds on the Malliavin derivative which are not depending on $b'$. 

In \emph{step 3} we consider a general $b$ and construct an approximation of the solution such that the sequence of Malliavin derivatives are bounded uniformly.

\bigskip

{\bf  Step 1}

The function $(\cdot)^- :\mathbb{R} \rightarrow \mathbb{R}$ is
defined by
\begin{equation*}
(y)^- = \left\{
\begin{array}{ll}
-y & \textrm{ if } y < 0\\
0 & \textrm{ if } y \geq 0  .\\
\end{array} \right.
\end{equation*}

Let $\rho \in C^{\infty}(\mathbb{R})$ be a positive function
such that supp$\{\rho\} \subset (-1,1)$ and $\int \rho(z) dz
=1$. Define $\rho_n(z) = n \rho(nz)$ and let
$$
h_n(y) = \int \rho_n(y-z) (z)^- dz . 
$$
It is readily checked that $h_n \in C^{\infty}(\mathbb{R})$,
$h_n'(z) \leq 0$ and $h_n \rightarrow (\cdot)^-$ almost
everywhere.  Then there exists a unique strong solution to the SDE
$$
X^{n, \epsilon}_t = x + \int_0^t b(X^{n,\epsilon}_s) +
\frac{1}{\epsilon} h_n (X_s^{n, \epsilon}) ds+ \int_0^t \sigma(X_s^{n,\epsilon}) dB_s .
$$
As $n \rightarrow \infty$ it is easy to see that  $X_t^{n, \epsilon}
\rightarrow X_t^{ \epsilon }$ in $L^2(\Omega)$, where
$$
X^{ \epsilon}_t = x + \int_0^t b(X^{\epsilon}_s) +
\frac{1}{\epsilon} (X_s^{ \epsilon})^- ds+ \int_0^t \sigma(X_s^{\epsilon}) B_s .
$$

The following lemma is a classical result. We include a proof for the sake for self-containdness.

\begin{lemma}
As $ \epsilon \rightarrow 0$, we get $(X_t^{\epsilon}, \epsilon^{-1}\int_0^t (X_s^{\epsilon})^- ds) \rightarrow (X_t, L_t)$ - the solution to (\ref{mainSDE}).
\end{lemma}

\begin{proof}
By the comparison principle, we note that there exists a subset with full measure $\Omega_0 \subset \Omega$ such that $\{X_t^{\epsilon}(\omega)\}$ is increasing as $\epsilon \rightarrow 0$ for all $(t,\omega) \in [0,1] \times \Omega_0$.
We may define
$$
X_t(\omega) = \lim_{\epsilon \rightarrow 0} X_t^{\epsilon}(\omega)
$$
for $\omega \in \Omega_0$ and $0$ otherwise. 

By It\^{o}'s formula we have
\begin{align} \label{approxSquare}
(X_t^{\epsilon})^2 & = x^2 + \int_0^t 2X_s^{\epsilon} (\epsilon^{-1} (X_s^{\epsilon})^- +  b(X_s)) + \sigma^2(X_s^{\epsilon}) ds + \int_0^t 2X_s^{\epsilon} \sigma(X_s^{\epsilon}) dB_s \\
\notag & \leq x^2 + \int_0^t 2X_s^{\epsilon}  b(X_s) + \sigma^2(X_s^{\epsilon}) ds + \int_0^t 2X_s^{\epsilon} \sigma(X_s^{\epsilon}) dB_s 
\end{align}
and taking expectation yields
\begin{align*}
E[(X_t^{\epsilon})^2] & \leq x^2 + \int_0^t E[ 2X_s^{\epsilon}  b(X_s)] + E[\sigma^2(X_s^{\epsilon})] ds \\
 & \leq x^2 + \|b\|^2_{\infty}t + \|\sigma\|^2_{\infty}t +  \int_0^t E[(X_s^{\epsilon})^2] ds \\
 & \leq (x^2 + \|b\|^2_{\infty}t + \|\sigma\|^2_{\infty}t) e^t \\
\end{align*}
where we have used the inequality $2 ab \leq a^2 + b^2$ and Gronwall's inequality.

It follows by Fatou's lemma that $X_t$ is $P$-a.s. finite.

Define $Y_t^{\epsilon}$ to be the solution of
$$
dY_t^{\epsilon} = \epsilon^{-1} (Y_t^{\epsilon})^- dt + b(X_t)dt + \sigma(X_t)dB_t, \, \, Y_0^{\epsilon} = x .
$$
From Proposition \ref{SkorohodApproximation} we get that on the subset on which $\int_0^{\cdot} b(X_s)ds + \int_0^{\cdot} \sigma(X_s)dB_s$ is continuous, we have that $Y^{\epsilon}$ (respectively $\epsilon^{-1} \int_0^{\cdot} (Y_s^{\epsilon})^- ds$) converges to $Y$ (respectively $\phi$) in $C([0,1])$ which is the solution to the Skorohod equation.

We have by It\^{o}'s formula,

\begin{align*}
(X_t^{\epsilon} - Y_t^{\epsilon})^2 & = \frac{2}{\epsilon} \int_0^t \left( (X_s^{\epsilon})^- - (Y_s^{\epsilon})^- \right) \left( X_s^{\epsilon} - Y_s^{\epsilon}\right) ds  \\
& + 2 \int_0^t \left( X_s^{\epsilon} - Y_s^{\epsilon}\right) \left( b(X_s^{\epsilon}) - b(X_s) \right) ds \\
& + 2 \int_0^t \left( X_s^{\epsilon} - Y_s^{\epsilon}\right) \left( \sigma(X_s^{\epsilon}) - \sigma(X_s) \right) dB_s \\
& + \int_0^t \left( \sigma(X_s^{\epsilon}) - \sigma(X_s) \right)^2 ds .\\ 
\end{align*}

The first term above is always negative. Taking expectation we get :
\begin{align*}
E[ (X_t^{\epsilon} - Y_t^{\epsilon})^2 ] & \leq 2 \int_0^t E |\left( X_s^{\epsilon} - Y_s^{\epsilon}\right) \left( b(X_s^{\epsilon}) - b(X_s) \right)| ds \\
& + \int_0^t E[ \left( \sigma(X_s^{\epsilon}) - \sigma(X_s) \right)^2 ]ds \\ 
& \leq   \int_0^t E [\left( X_s^{\epsilon} - Y_s^{\epsilon}\right)^2] ds + \int_0^t E[\left( b(X_s^{\epsilon}) - b(X_s) \right)^2] ds \\
& + \int_0^t E[ \left( \sigma(X_s^{\epsilon}) - \sigma(X_s) \right)^2] ds \\ 
& \leq   e^t \left(  \int_0^t E[\left( b(X_s^{\epsilon}) - b(X_s) \right)^2] ds  + \int_0^t E[ \left( \sigma(X_s^{\epsilon}) - \sigma(X_s) \right)^2] ds \right).\\ 
\end{align*}

Above we have used Gronwalls lemma in the last inequality. As $\epsilon \rightarrow 0$ the above goes to zero, and we see that $Y_t^{\epsilon} \rightarrow X_t$ $P$-a.s. for all $t \in [0,1]$. 
Using the Burkholder-Davies-Gundy inequality, one can show that this convergence actually takes place in $L^2(\Omega; C([0,1]))$.

It follows from Proposition \ref{SkorohodApproximation} that $X_t$ is continuous and that $\epsilon^{-1} \int_0^{\cdot} (X_s^{\epsilon})^- ds$ converges to $L_{\cdot}$ where $(X,L)$ is a solution (\ref{mainSDE}).

\end{proof}

{\bf Step 2 }

We have the following estimate on the Malliavin derivatives:

\begin{lemma} \label{firstEstimateLemma}
For fixed $t \geq 0$ we have $X_t \in \mathbb{D}^{1,2}$ and there exists an increasing function $K_1 : \mathbb{R}_+ \rightarrow \mathbb{R}_+$ such that the
Malliavin derivative satisfies
$$
E[(D_{\theta}X_t)^2] \leq K_1(\| \sigma \|_{C^1_b}) \left(E[ \exp \{ 4\int_{\theta}^t b'(X_s)ds \}] \right)^{1/2}.
$$
\end{lemma}

\begin{proof}
We observe that $X_t^{n, \epsilon} \in \mathbb{D}^{1,2}$ and the
Malliavin derivative satisfies
\begin{equation} \label{MalliavinDerivative}
D_{\theta}X_t^{n, \epsilon} = \sigma(X_{\theta}^{n,\epsilon}) + \int_{\theta}^t (b'(X_s^{n,
\epsilon}) + \epsilon^{-1} h_n'(X_s^{n,\epsilon}))
D_{\theta}X_s^{n, \epsilon} ds  + \int_{\theta}^t \sigma'(X_s^{n,\epsilon}) D_{\theta}X_s^{n, \epsilon} dB_s.
\end{equation}
This is a linear SDE which is uniquely solved by
\begin{align*}
D_{\theta}X_t^{n, \epsilon} & = \sigma(X_{\theta}) \exp \left\{ \int_{\theta}^t
b'(X_s^{n, \epsilon}) + \epsilon^{-1} h_n'(X_s^{n, \epsilon})  - \frac{1}{2} \left( \sigma'(X_s^{n, \epsilon} \right)^2ds  \right\} \\
& \times \exp \{ \int_{\theta}^t \sigma'(X_s^{n, \epsilon}) dB_s \} \\
&  \leq \exp \{ \int_{\theta}^t b'(X_s^{n, \epsilon})  ds \} \exp \{ \int_{\theta}^t \sigma'(X_s^{n, \epsilon}) dB_s \}  \\
\end{align*}
since $h_n'$ is negative.

Using that for a bounded adapted process $\{ \psi(s) \}_{s \in [0,1]}$ we have
$$
E[ \exp \{ \int_{\theta}^t \psi(s) dB_s \} ] \leq \exp \left\{ \frac{(t- \theta)}{2} \| \psi\|_{\infty}^2 \right\}
$$
and H\"{o}lder's inequality, we get
$$
E[(D_{\theta}X_t^{n,\epsilon})^2] \leq \|\sigma\|_{\infty}^2 \exp \{ c \| \sigma' \|^2_{\infty} \} \left(E[ \exp \{ 4\int_{\theta}^t b'(X_s^{n, \epsilon})ds \}] \right)^{1/2}.
$$

Letting first $n$ go to infinity and $\epsilon$ tend to zero we
get the result.

\end{proof}

What is left is to find a bound on $E[ \exp \{ 4\int_{\theta}^t b'(X_s^{n, \epsilon})ds \}]$ that is depending only on $\|b\|_{\infty}$.

The following Proposition is based on Proposition 3 in \cite{Davie2011}.

\begin{prop} \label{DavieEstimate}
There exists a constant $C$ such that for every positive integer $k$ we have
$$
E \left( \int_{\theta}^t b'(X_s)ds \right)^k \leq \frac{C^k \|b\|_{\infty}^k |t - \theta|^{k/2} k!}{\Gamma( \frac{k}{2}+1)}
$$
for all $b \in C_c^{\infty}((0, \infty))$.
\end{prop}

Let us explain briefly the idea of the proof. Using the Markov property we can write the above left-hand-side as
$$
\int_{\theta < t_1 < \dots < t_k < t} \int_{\mathbb{R}_+^k} \prod_{j=1}^k b'(z_j)P(t_j-t_{j-1}, z_j,z_{j-1})  dz_k \dots dz_1 dt_1 \dots dt_k ,
$$
where $P$ is the transition density of (\ref{mainSDE}). Then use integration by parts to move the derivatives onto the density function. Then one can show the result by using estimates on $P$ and its derivatives.

Let us remark that the proof of Proposition \ref{DavieEstimate} is the same as the proof of Proposition 3 in \cite{Davie2011} when we replace Lemma 1 in \cite{Davie2011} by the following:

\begin{lemma} \label{boundedOperator}
The operator $T : L^2([0,1] \times \mathbb{R}_+)  \rightarrow L^2([0,1] \times \mathbb{R}_+)$ defined by
$$
Th(s,y) = \int_s^1 \int_{\mathbb{R}_+} \partial_x \partial_y P(t-s, y,z)  h(t,z) dz dt
$$
is bounded.
\end{lemma}

We note that $P(t,x,y)$ the fundamental solution to
$$
\partial_t u = b \partial_x u + \frac{1}{2} \sigma^2 \partial^2_x u, \, \, \, \partial_x u |_{t=0} = 0 .
$$

Lemma \ref{boundedOperator} follows from a 'T(1) theorem on spaces of homogeneous type' using the Schauder estimates obtained in the following lemma:

\begin{lemma}
We equip $[0,1] \times \mathbb{R}$ with the parabolic metric $d(t,x) = \sqrt{t} + |x|$. 
There exists constants $C$, $c >0$ such that we have
\begin{itemize}

\item
$
|P(t,x,y)| \leq Ct^{-1/2} \exp\{ \frac{ -c(x-y)^2}{t} \}
$

\item

$
| \partial_x \partial_y P(t,x,y) | \leq C t^{-3/2} \exp \{ \frac{-c(x-y)^2}{t} \}
$

\item

$
| \partial_x \partial_y P(t,x,y) | \leq d(t,x-y)^{-3}
$

\item
$
| \partial_x \partial_y P(t-s,x,y) - \partial_x \partial_y P(t' - s, x,y') | \leq C \frac{d(t-t', y-y')^{\delta}}{d(t-s,x-y)^{3 + \delta}}
$
for some $\delta > 0$, whenever $\frac{d(t-t', y-y')}{d(t-s,x-y)} < \frac{1}{2}$.

\item
$
\int_{\mathbb{R}} \partial_x \partial_y P(t,x,y) dy = \int_{\mathbb{R}} \partial_x \partial_y P(t,x,y) dx = 0
$
\end{itemize}

\end{lemma}

Combining Poposition \ref{DavieEstimate} and Lemma \ref{firstEstimateLemma} we are able finish step 1:

\begin{prop} \label{mainEstimate}
There exists a continuous function 
$C_{\delta}: \mathbb{R}^2_+ \rightarrow \mathbb{R}_+$ increasing in both variables such that
$$
E[(D_{\theta}X_t)^2] \leq C_{\delta}(\|b\|_{\infty}, \|\sigma\|_{C^2_b}) .
$$
Moreover, $C_{\delta}$ is independent of $t$ and $\theta$.

\end{prop}

\begin{remark}

By approximation, one can get the same estimate as in Proposition \ref{mainEstimate} when assuming that $b$ is Lipschitz continuous. 

\end{remark}

We now turn to step 3 of our proof, which is concludes the proof of Theorem \ref{MalliavinThm}.

\bigskip

\textbf{ Step 3 }

\begin{proof}[Proof of Theorem \ref{MalliavinThm}]

Assume $b : \mathbb{R} \rightarrow \mathbb{R}$ is bounded and measurable. 
Choose a function $\psi \in C^{\infty}$ such that 
$$
\psi(y) = 
\left\{  \begin{array}{ll}
1 &  \textrm{ if } y \geq 1\\
0 &  \textrm{ if } y \leq 0 \\
\end{array} \right. .
$$
For $n \in \mathbb{N}$ we define $\psi_n^0(y) = \psi(ny)$, $\psi^1_n(y) = 1 - \psi(n^{-1}y - n)$
and $\psi_n(y) = \psi_n^0(y) + \psi_n^1(y)$. It is readily checked that $\psi_n$ is smooth and has compact support. Moreover, $\psi_n(0) =0$ and 

$$
\lim_{n \rightarrow \infty} \psi_n(y) =
\left\{  \begin{array}{ll}
1 &  \textrm{ if } y > 0\\
0 &  \textrm{ if } y \leq 0 \\
\end{array} \right. .
$$

Define $b_j(y) := \int \rho_j(y-z) b(z)dz \psi_n(y)$, and let
$$
b_{n,k} := \bigwedge_{j=1}^k b_j,
$$
and 
$$
\hat{b}_n = \bigwedge_{j=n}^{\infty} b_j .
$$

Then $b_{n,k}$ is Lipschitz continuous, $b_{n,k}(0) = 0$, $\hat{b}_n$, $b_{n,k}$ are uniformly bounded and we have
$$
b_{n,k} \geq b_{n,k+1} \geq \dots \rightarrow \hat{b}_n, \, \, \textrm{ as } k \rightarrow \infty, 
$$
and
$$
\hat{b}_n \leq \hat{b}_{n+1} \rightarrow b \, \, \textrm{ as } n \rightarrow \infty
$$
almost surely with respect to Lebesgue measure.

Using the comparison theorem for SDE's one can show that for the corresponding sequences of solutions, denoted $(X^{n,k}, L^{n,k})$ and $(X^n, L^n)$ , we have the following convergence in $L^2(\Omega)$:
$$
(X_t^n,L^n_t) = \lim_{k \rightarrow \infty} (X_t^{n,k}, L_t^{n,k}) \, \, \textrm{ uniformly in } t
$$
and 
$$
(X_t,L_t) = \lim_{n \rightarrow \infty} (X_t^{n}, L_t^{n}) \, \, \textrm{ uniformly in } t
$$
where $(X,L)$ is a solution to (\ref{mainSDE}). Details can be found in \cite{Tusheng94}.

By Proposition \ref{mainEstimate} we have $\sup_{n,k \geq 1} \| X_t^{n,k}\|_{1,2} < \infty$. The result follows.

\end{proof}

\section{Spatial Regularity} \label{spatialRegularity}

In this section we want to emphazise that the equation (\ref{mainSDE}) depends on the initial value $x \geq 0$. We write $X_t(x)$ for the unique strong solution.

\begin{prop} \label{WeakDifferentiability}

The solution to (\ref{mainSDE}) is locally weakly differentiable in the sense that for a bounded, open subset $U \subset \mathbb{R}_+$ and any $p > 1$ we have
$$
X_t(\cdot) \in L^2(\Omega ; W^{1,p}(U)) .
$$

\end{prop}

The proof follows the same steps as in the previous section and we just indicate the proof here.

For the first step we assume $b \in C^1_b(\mathbb{R})$, $b(0) =0$, and we consider the approximating sequence of solutions
$$
X^{n, \epsilon}_t(x) = x + \int_0^t b(X^{n,\epsilon}_s(x)) +
\frac{1}{\epsilon} h_n (X_s^{n, \epsilon}(x)) ds+ \int_0^t \sigma(X_s^{n,\epsilon}(x)) dB_s .
$$

Then the solution is in $C^1$ and we have that the spatial derivative satisfies
\begin{align*}
\partial_x X^{n, \epsilon}_t(x) & = 1 + \int_0^t b'(X^{n,\epsilon}_s(x))\partial_x X^{n, \epsilon}_s(x) ds \\ 
& + \int_0^t\frac{1}{\epsilon} h_n' (X_s^{n, \epsilon}(x)) \partial_x X^{n, \epsilon}_s(x)ds+ \int_0^t \sigma'(X_s^{n,\epsilon}(x)) \partial_x X^{n, \epsilon}_s(x)dB_s . \\
\end{align*}

We recognize this equation as the same as (\ref{MalliavinDerivative}) when we let $\theta = 0$. It is then easy to see that the results of Lemma \ref{firstEstimateLemma}, Propositions \ref{DavieEstimate} and \ref{mainEstimate} when we replace the Malliavin derivative by the spatial derivative.

More specifically, we in place of Proposition \ref{mainEstimate} we get that when $b$ is Lipschitz,
$$
\sup_{x \geq 0} E| \partial_x X_t(x)|^p \leq C(\|\sigma\|_{C^1_b}, \|b\|_{\infty}) .
$$

Since $U \subset \mathbb{R}_+$ is bounded we see that $X_t(\cdot) \in L^2(\Omega; W^{1,p}(U))$.

If now $b$ is merely bounded and measurable we use the same method as step 2 in the previous section to conclud:

\begin{lemma} \label{WeakDifferentiabilityApproximation}
There exists a sequence $X_t^k(\cdot)$, bounded in $L^2(\Omega ; W^{1,p}(U)$, such that $X_t^k(x) \rightarrow X_t(x)$ in $L^2(\Omega)$ for all $x \geq 0$.
\end{lemma}

We arrive at the proof of Proposition \ref{WeakDifferentiability}

\begin{proof}[ Proof of \ref{WeakDifferentiability} ]
From Lemma \ref{WeakDifferentiabilityApproximation} we get that there exists a subsequence $\{X_t^{k_j}( \cdot) \}_{ j \geq 1}$ that is converging in the weak topology of $L^2(\Omega; W^{1,p}(U)$ to some element $Y_t$. Since $X_t^{k_j}(x) \rightarrow X_t(x)$, we have for any $A \in \mathcal{F}$ and $\varphi \in C^{\infty}_c(U)$
\begin{align*}
E[ 1_A \int_U \varphi'(x) X_t(x) dx ] & = \lim_{j \rightarrow \infty} E[ 1_A \int_U \varphi'(x) X_t^{k_j}(x) dx ] \\
 &=  - \lim_{j \rightarrow \infty} E[ 1_A \int_U \varphi(x) \partial_x X_t^{k_j}(x) dx ] \\
&=  - E[ 1_A \int_U \varphi(x) \partial_x Y_t(x) dx ] .\\
\end{align*}

It follows that $X_t(\cdot)$ is $P$-a.s. weakly differentiable and it's weak derivative is equal to $\partial_x Y_t(x)$.

\end{proof}

\section{Bismut-Elworthy-Li Formula} \label{bismutFormula}

In this section we study the PDE
\begin{equation} \label{PDE}
\partial_t u(t,x) = b(x) \partial_x u(t,x) + \frac{1}{2} \sigma^2(x) \partial_x^2 u(t,x), \textrm{ for } x \geq 0
\end{equation}

with initial and boundary condition
$$
u(0,x) = u_0(x), \, \, \, \partial_xu(t,0) = 0 .
$$

We shall use the same assumptions on $b$ and $\sigma$ as in Theorem \ref{MalliavinThm} and $u_0 \in C^1_b(\mathbb{R}_+)$.

Existence and uniqueness of a solution to (\ref{PDE}) is already known. More specifically, the solution is given by 
$$
u(t,x) = E[ u_0(X_t(x))]
$$
and lies in $W^{(1,2),p}_{loc}((0,1] \times \mathbb{R}_+)$ - the space of functions which are once weakly differentiable w.r.t $t$ and twice weakly differentiable w.r.t. $x$ and these functions are locally $p$-integrable. Moreover, the solution is in $C([0,1] \times \mathbb{R}_+)$.

In this section, however, we shall prove a Bismut-Elworthy-Li formula for the derivative of the solution to (\ref{PDE}) which does not depend on the derivative of $u_0$.

\begin{thm}
For a bounded subset $U \subset \mathbb{R}_+$ the (weak) spatial derivative of $u$ takes the form
\begin{equation} \label{BismutFormula}
\partial_x u(t,x) = E[ u_0(X_t(x)) t^{-1} \int_0^t \partial_x X_s(x) dB_s ] ,
\end{equation}
for almost every $x \in U$.

\end{thm}

\begin{proof}
As in the proof of Proposition \ref{WeakDifferentiability} we have a sequence of processes $\{ X_t^k(x) \}$ that are $P$-a.s. differentiable in $x$, $X_t^k (x) \rightarrow X_t(x)$ in $L^2(\Omega)$ and $X_t^k(\cdot) $ converges to $X_t(\cdot)$ in the weak topology of $L^2(\Omega; W^{1,p}(U))$.

We certainly get that 
$$
u_k(t,x) := E[u_0(X_t^k(x))] \rightarrow u(t,x)
$$
as $k \rightarrow \infty$ for every $(t,x) \in [0,1] \times \mathbb{R}_+$. We will now show that $\partial_xu_k(t,\cdot)$ converges weakly to the right-hand-side of (\ref{BismutFormula}), thus proving the assertion.

We start by noting that $\partial_x X^k_t(x) = D_sX^k_t(x) \partial_x X^k_s(x)$, and by the chain-rule for the Malliavin derivative we have
\begin{align*}
u_0'(X_t^k(x)) \partial_x X_t^k(x) & = u_0'(X_t^k(x)) t^{-1}\int_0^t D_sX^k_t(x) \partial_x X^k_s(x) ds \\
 & =t^{-1} \int_0^t D_s ( u_0(X_t^k(x))) \partial_x X_s^k(x) ds .\\
\end{align*}
Taking expecations in the above formula and using the duality between the Malliavin derivative and the It\^{o}-integral we get
\begin{align*}
\partial_x u_k(t,x) & = E[u_0'(X_t^k(x)) \partial_x X_t^k(x)] \\
& = t^{-1} E[ \int_0^t D_s( u_0(X_t^k(x))) \partial_x X_s^k(x) ds ] \\
& = t^{-1} E[ u_0(X_t^k(x)) \int_0^t  \partial_x X_s^k(x) dB_s ] .\\
\end{align*}

For a test function $\varphi \in C^{\infty}_c(U)$ we have
\begin{align*}
\int_U \varphi'(x) u(t,x)dx   &= - \lim_{k \rightarrow \infty}  \int_U \varphi(x) \partial_x u_k(t,x)dx\\
 & = - \lim_{k \rightarrow \infty} \int_U \varphi(x) t^{-1} E[ u_0(X_t^k(x)) \int_0^t  \partial_x X_s^k(x) dB_s ]dx .\\
&  = - \lim_{k \rightarrow \infty} \int_U \varphi(x) t^{-1} E[ (u_0(X_t^k(x))  - u_0(X_t(x))) \int_0^t  \partial_x X_s^k(x) dB_s ]dx .\\
 & \, \,  \, \, \, \, - \lim_{k \rightarrow \infty} \int_U \varphi(x) t^{-1} E[ u_0(X_t(x)) \int_0^t  \partial_x X_s^k(x) dB_s ]dx. \\
\end{align*}

To see that the first term converges to zero, note that for all $x \in U$
\begin{align*}
& E[ (u_0(X_t^k(x))  - u_0(X_t(x))) \int_0^t  \partial_x X_s^k(x) dB_s ] \\ \leq &  \|u'\|_{\infty} \| X_t^k(x) - X_t(x) \|_{L^2(\Omega)} \| \int_0^t \partial_x X_s^k(x) dB_s \|_{L^2(\Omega)} \\
  = & \|u'\|_{\infty} \| X_t^k(x) - X_t(x) \|_{L^2(\Omega)} \left( \int_0^t E[ | \partial_x X_s^k(x) |^2 ] ds \right)^{1/2} \\
\leq & \|u'\|_{\infty} \| X_t^k(x) - X_t(x) \|_{L^2(\Omega)} \left( t \sup_{m,y,s}E[ | \partial_x X_s^m(y) |^2 ]  \right)^{1/2} \\
\end{align*}
which converges to zero as $ k \rightarrow \infty$.

For the second term, notice that since $X_t(x)$ is Malliavin differentiable and $u_0 \in C^1_b(\mathbb{R}_+)$, we have by the Clark-Ocone formula
\begin{equation*} 
u_0(X_t(x)) = E [u_0(X_t(x))] + \int_0^t E [ D_s u_0(X_t(x)) | \mathcal{F}_s ] dB_s
\end{equation*}
and so 
\begin{align*}
\lim_{k \rightarrow \infty} \int_U \varphi(x) t^{-1} E[ u_0(X_t(x)) \int_0^t  \partial_x X_s^k(x) dB_s ]dx   \\
= \lim_{k \rightarrow \infty} \int_U \varphi(x) t^{-1} E[\int_0^t D_s (u_0(X_t(x)))  \partial_x X_s^k(x) ds ]dx   \\
 = \int_U \varphi(x) t^{-1} E[\int_0^t D_s (u_0(X_t(x)))  \partial_x X_s(x) ds ]dx   \\
\end{align*}
where we have used dominated convergence w.r.t $s$ and weak convergence w.r.t. $x$.

We finally note that 
$$
E[\int_0^t D_s (u_0(X_t(x)))  \partial_x X_s(x) ds ] = E[ u_0(X_t(x)) \int_0^t \partial_xX_s(s) dB_s ]
$$
again by the Clark-Ocone formula. 

\end{proof}

\section{Appendix: The Skorohod equation} \label{skorohodEquation}

Given a continuous function $g$ such that $g(0) = 0$ and $x \geq 0$ we are searching for nondecreasing function $\phi \in C([0,1])$ such that if we \emph{ define }
$$
f(t) := x + \phi(t) + g(t),
$$
then $f(t) \geq 0$ for all $t \in C([0,1])$ and $\int_0^1 1_{\{ f(s) > 0\}} d\phi(s) = 0$. We call the pair $(f,\phi)$ a solution to the Skorohod equation if they satisfies the above.

It is well known that such a solution exists and it is uniquely given by
$$
\phi(t) = \max \{ 0, \max_{0 \leq s \leq t} -(x + g(s)) \} .
$$

The topic of this Appendix is however to approximate the solution in a suitable sense.

Let $\epsilon > 0$ and denote by $f^{\epsilon}$ the solution of the following ODE:
\begin{equation} \label{ODE}
f^{\epsilon}(t) = x + \frac{1}{\epsilon}\int_0^t (f^{\epsilon}(s))^- ds + g(t) .
\end{equation}

We have:

\begin{lemma} \label{DeterministicConvergence}
Assume $g \in C^1([0,1])$, $g(0)=0$. As $\epsilon \rightarrow 0$, there exists a subsequence of $(f^{\epsilon}, \epsilon^{-1} \int_0^{\cdot} (f^{\epsilon}(s))^-ds)$  converging uniformly to $(f, \phi)$ - the solution to the Skorohod equation.

\end{lemma}

\begin{proof}
By the comparison principle for ODE's we have that $f^{\epsilon}$ is pointwise increasing to some function denoted $f$. We have
\begin{align*}
(f^{\epsilon}(t))^2  & = x^2 + \frac{2}{\epsilon} \int_0^t f^{\epsilon}(s) (f^{\epsilon}(s))^- ds + 2\int_0^t f^{\epsilon}(s) \dot{g}(s) ds \\
 & \leq x^2 + \int_0^t (f^{\epsilon}(s))^2 +  (\dot{g}(s))^2 ds \\
& \leq \left( x^2 + \int_0^t (\dot{g}(s))^2 ds \right)e^t \\
\end{align*}
where the last inequality comes from Gronwalls lemma. It follows that $f(t) < \infty$ for all $t$.

Furthermore, we have
\begin{align*}
(f^{\epsilon}(t))^-  & =  \frac{1}{\epsilon} \int_0^t 1_{(f^{\epsilon} \leq 0)}  (f^{\epsilon}(s))^- ds + \int_0^t 1_{(f^{\epsilon} \leq 0)}  \dot{g}(s) ds \\
 & =  \frac{1}{\epsilon} \int_0^t (f^{\epsilon}(s))^- ds + \int_0^t 1_{(f^{\epsilon} \leq 0)}  \dot{g}(s) ds \\
& = \int_0^t 1_{(f^{\epsilon} \leq 0)}  \dot{g}(s) e^{-\epsilon^{-1}(t-s)} ds 
\end{align*}
where the last inequality comes from solving the linear ODE that $(f^{\epsilon}(t))^-$ satisfies. It follows from the above that
$$
(f^{\epsilon}(t))^- \leq \|\dot{g}\|_{\infty} \epsilon .
$$
From (\ref{ODE}) and the above shows that 
$$
\dot{f}^{\epsilon}(t) = \epsilon^{-1} f^{\epsilon}(t) + \dot{g}(s)
$$
is then uniformly bounded by $2 \|\dot{g}\|_{\infty}$. Using the relative compactness of $\dot{f}^{\epsilon}$ in $L^2([0,1])$ with respect to the weak topology we can extract a converging subsequence (still denoted $\dot{f}^{\epsilon}$ for simplicity). Denote the limit by $\tilde{f}$.

Then,
$$
f(t)  = \lim_{\epsilon \rightarrow 0} f^{\epsilon}(t)  = x + \lim_{\epsilon \rightarrow 0} \int_0^1 1_{[0,t]} (s) \dot{f}^{\epsilon}(s) ds  = x +  \int_0^1 1_{[0,t]} (s) \tilde{f}(s) ds
$$
so that $f$ is continuous. It follows from Dini's theorem that the convergence is uniform in $t$.

To see that $f(t)$ is positive assume that there exists $t_0$ such that $f(t_0) < 0$. By continuity we may choose $\delta > 0$ such that $f(t) \leq \frac{f(t_0)}{2}$ for all $t \in (t_0 - \delta, t_0 + \delta)$. Moreover, by the uniform convergence there exists $\epsilon_0 > 0$ such that 
$$
f^{\epsilon}(t) \leq \frac{f(t_0)}{4}, \, \forall t \in (t_0 - \delta, t_0 + \delta)  \textrm{ and } \forall \epsilon < \epsilon_0.
$$
It follows that 
\begin{align*}
f^{\epsilon}(t_0 + \frac{\delta}{2}) - f^{\epsilon}(t_0 - \frac{\delta}{2}) & = \int_{t_0 - \delta/2}^{t_0 + \delta/2} (f^{\epsilon}(s))^- ds + g(t_0 + \frac{\delta}{2}) - g(t_0 - \frac{\delta}{2}) \\
& \geq - \frac{f(t_0) \delta}{4\epsilon} +  g(t_0 + \frac{\delta}{2}) - g(t_0 - \frac{\delta}{2}) \\
& \rightarrow +\infty \\
\end{align*}
as $\epsilon \rightarrow 0$ which contradicts the finiteness of $f$. Consequently, $f(t) \geq 0$ for all $t$.

It is clear from (\ref{ODE}) that also $\epsilon^{-1} \int_0^{\cdot} (f^{\epsilon}(s))^- ds$ is converging in $C([0,1])$, and we denote the limit by $\phi(t)$. Being the limit of a sequence of nondecreasing functions, $\phi$ itself is increasing. Moreover, we have that $\phi$ is constant on $\{ t \in [0,1] | f(t) > 0 \}$. Indeed, assume $f(t) > \gamma > 0$ for all $t \in (a,b) \subset [0,1]$. We may choose $\epsilon_0>0$ such that $f^{\epsilon}(t) \geq \gamma/2>0$ for all $t \in (a,b)$. 

For $r<s \in (a,b)$ we have
$$
\phi(s) - \phi(r) = \lim_{\epsilon \rightarrow 0} \epsilon^{-1} \int_r^s (f^{\epsilon}(u))^- du = 0 ,
$$
the claim follows and we get $\int_0^1 1_{(0, \infty)}(f(s)) d\phi(s)  = 0$.

\end{proof}

We write $f_g$ to emphasize that the above function depends on $g$. We can then get the following

\begin{lemma} \label{DeterministicExtension}
The mapping $g \mapsto f_g$ is can be extended to a Lipschitz-continuous mapping from $C([0,1])$ into itself.
\end{lemma}

\begin{proof}
If we denote by $f_j^{\epsilon}$ the solution to (\ref{ODE}) when we replace $g$ by $g_j \in C^1([0,1])$, $j=1,2$, it is enough to find the uniform bound 
\begin{equation} \label{ODEbound}
\|f_1^{\epsilon} - f_2^{\epsilon} \|_{\infty}  \leq 2 \|g_1 - g_2\|_{\infty} .
\end{equation}

To this end, define the functions 
$$
h_j^{\epsilon} (t) := f_j^{\epsilon}(t) - g_j(t) = x + \int_0^t \left( h_j^{\epsilon}(s) + g_j(s) \right)^- ds .
$$
With $K := \|g_1 - g_2\|_{\infty}$ we have
$$
\left( h_1^{\epsilon}(t) - h_2^{\epsilon}(t) - K \right)^+ 
$$
$$
= \int_0^t 1_{ \left(h_1^{\epsilon}(s) - h_2^{\epsilon}(s) > K \right)} \left(  (h_1^{\epsilon}(s) + g_1(s) )^- - ( h_2^{\epsilon}(s) + g_2(s))^- \right) ds .
$$

We see that the above integrand is negative for all $s$. Indeed, fix $s \in [0,t]$. 
If $(h_1^{\epsilon}(s) + g_1(s) )^- - ( h_2^{\epsilon}(s) + g_2(s))^-$ is negative we are done. 
If $(h_1^{\epsilon}(s) + g_1(s) )^- - ( h_2^{\epsilon}(s) + g_2(s))^-$ is positive, we write
$$
h_1^{\epsilon}(s) - h_2^{\epsilon}(s) \leq h_1^{\epsilon}(s) + g_1(s)  -  h_2^{\epsilon}(s) + g_2(s) + K
$$
so that
$$
1_{ \left(h_1^{\epsilon}(s) - h_2^{\epsilon}(s) > K \right)} \leq 1_{ \left(h_1^{\epsilon}(s) + g_1(s) - h_2^{\epsilon}(s) + g_2(s) > 0 \right)} .
$$
It is then easy to check that the function $(x,y) \mapsto 1_{(x-y) > 0} ((x)^- - (y)^-)$ is always non-positive.

\end{proof}

We are ready to conclude:

\begin{prop} \label{SkorohodApproximation}
Let $g \in C([0,1])$ be such that $g(0)=0$ and $x \geq 0$. Then the solution to 
$$
f_g^{\epsilon}(t) = x + \epsilon^{-1} \int_0^t (f_g^{\epsilon}(s))^- ds + g(t)
$$
converges in $C([0,1])$ as $\epsilon \rightarrow 0$ to the solution to the Skorohod equation.

\end{prop}

\begin{proof}
Let $\delta > 0$. Choose $g_{\delta} \in C^1([0,1])$ such that $\| g -g_{\delta}\|_{\infty} < \delta$. By Lemma \ref{DeterministicConvergence} we can choose $\epsilon_0 > 0$ such that for all $\epsilon \in (0, \epsilon_0)$ we have $ \| f_{g_{\delta}} - f^{\epsilon}_{g_{\delta}} \|_{\infty} < \delta$. By the proof of Lemma \ref{DeterministicExtension} we get
\begin{align*}
\|f_g - f_g^{\epsilon}\|_{\infty} & \leq \| f_g - f_{g_{\delta}}\|_{\infty} + \| f_{g_{\delta}} - f_{g_{\delta}}^{\epsilon} \|_{\infty} + \| f_{g_{\delta}}^{\epsilon} - f_g^{\epsilon} \|_{\infty} \\
& < 2 \|g - g_{\delta}\| + \delta + 2 \|g - g_{\delta}\|_{\infty} = 5 \delta .
\end{align*}
The conditions of the Skorohod equation are easy to check.

\end{proof}

\end{document}